\documentclass[a4paper, 12pt]{article}
\usepackage{mathrsfs}
\usepackage{amsmath}
\usepackage{amsfonts}
\usepackage[T1]{fontenc}
\usepackage[latin9]{inputenc}
\usepackage{amsthm}
\usepackage{graphicx}
\usepackage{amsmath}

\usepackage{amsthm}
\usepackage{array}
\usepackage{cases}
\makeatletter
\def\th@plain{%
  \upshape 
}
\makeatother

\makeatletter
\renewenvironment{proof}[1][\proofname]{\par
  \pushQED{\qed}%
  \normalfont \topsep6\p@\@plus6\p@\relax
  \trivlist
  \item[\hskip\labelsep
        \bfseries
    #1\@addpunct{.}]\ignorespaces
}{%
  \popQED\endtrivlist\@endpefalse
}
\makeatother

\numberwithin{equation}{section}

\newtheorem{thm}{Theorem}[section]

\newtheorem{lem}[thm]{Lemma}

\newtheorem{conj}[thm]{Conjecture}

\newtheorem{clm}{Claim}

\numberwithin{equation}{section}

\usepackage[varg]{txfonts}
\usepackage{graphicx, epsfig, subfigure}
\usepackage{enumerate} 
\usepackage[square, numbers, sort&compress]{natbib} 
\ifx\pdfoutput\undefined
 \usepackage[dvipdfm,%
  pdfstartview=FitH, 
  bookmarks=true,%
  bookmarksnumbered=true, 
  bookmarksopen=true, 
  plainpages=false,%
  pdfpagelabels,%
  colorlinks=true, 
  linkcolor=blue, 
  citecolor=blue,%
  urlcolor=black,
  pdfborder=001]{hyperref}
  \AtBeginDvi{}  
\else
 \usepackage[pdftex,%
  pdfstartview=FitH, 
  bookmarks=true,%
  bookmarksnumbered=true, 
  bookmarksopen=true, 
  plainpages=false,%
  pdfpagelabels,%
  colorlinks=true, 
  linkcolor=blue, 
  citecolor=blue,%
  urlcolor=black,
  pdfborder=001]{hyperref}
\fi

\numberwithin{equation}{section}

\begin{document}

\title{Equitable partition of plane graphs with independent crossings into induced forests\footnotetext{Emails: B. Niu (beiniu@stu.xidian.edu.cn), X. Zhang (xzhang@xidian.edu.cn), Y. Gao (gaoyp@lzu.edu.cn)}}
\author{Bei Niu$^1$, \,\,\,\, Xin Zhang$^1$\thanks{Corresponding author.}~\thanks{Supported by the National Natural Science Foundation of China (11871055) and the Youth Talent Support Plan of Xi'an Association for Science and Technology (2018-6).}, \,\,\,\,Yuping Gao$^2$\thanks{Supported by the National Natural Science Foundation of China (11901263).}\\
{\small 1. School of Mathematics and Statistics, Xidian University, Xi'an 710071, China}\\
{\small 2. School of Mathematics and Statistics, Lanzhou University, Lanzhou 730000, China}
}

\maketitle

\begin{abstract}\baselineskip  0.6cm
The cluster of a crossing in a graph drawing in the plane is the set of the four end-vertices of
its two crossed edges. Two crossings are independent if their clusters do not intersect. In this paper, we prove that every plane graph with independent crossings has an equitable partition into $m$ induced forests for any $m\geq 8$. Moreover, we decrease this lower bound 8 for $m$ to 6, 5, 4 and 3 if we additionally assume that the girth of the considering graph is at least 4, 5, 6 and 26, respectively.

\end{abstract}

\baselineskip  0.6cm

\section{Introduction}

All graphs considered in this paper are finite and simple unless otherwise stated. By $V(G)$, $E(G)$, $\delta(G)$ and $\Delta(G)$, we denote the vertex set, the edge set, the minimum degree and the maximum degree of a graph $G$, respectively. In this paper, $|G|$ stands for $|V(G)|$, and $e(G)$ stands for $|E(G)|$. For two disjoint subsets  $S_1$ and $S_2$  of $V(G)$, $E(S_1,S_2)$ (resp.\,$e(S_1,S_2)$) is the set (resp.\,number) of edges that have one end-vertex in $S_1$ and another in
$S_2$. Under this notation, if $S_1$ consists of only one vertex $v$, then we use $e(v,S_2)$ instead of $e(\{v\},S_2)$.
The \emph{girth} $g(G)$ of a graph $G$ is the length of the shortest cycle in $G$, and is $+\infty$ if $G$ is a forest.
For other undefined notation, we refer the readers to \cite{Bondy.2008}.

An \emph{equitable partition} of a graph $G$ is a partition of $V(G)$ such that the sizes of any two parts differ by at most one. In 1970, Hajnal and Szemer\'edi \cite{Hajnal} answered a question of Erd\H{o}s by proving that every graph $G$ with maximum degree $\Delta$ has an equitable partition into $m$ independent sets for any integer $m\geq \Delta+1$.

Note that a star with maximum degree $\Delta$ has an equitable partition into $m$ stable sets for any $m\geq \lceil\frac{\Delta}{2}\rceil+1$, but it admits no equitable partition into $m$ independent sets for any $m<\lceil\frac{\Delta}{2}\rceil$. Therefore, finding a constant $c$ such that every planar graph has an equitable partition into $m$ independent sets for any $m\geq c$ is impossible. Surprisingly, if we ask for an equitable partition into induced forests rather than stable sets, we succeed. In 2005, Esperet, Lemoine and Maffray \cite{ELM.2015} confirmed a conjecture of Wu, Zhang and Li \cite{WZL.2013} by proving the following theorem.

\begin{thm}
Every planar graph has an equitable partition into $m$ induced forests for any $m\geq 4$.
\end{thm}

An open problem here is to determine whether every planar graph has an equitable partition into three induced forests (partial results on this problem can be found in \cite{Z2015}). If it is so, this number three is sharp.

A graph is \emph{1-planar} if it can be drawn in the plane so that each edge is crossed by at most one other edge, and a drawing satisfying this property so that the number of crossings is as few as possible is a \emph{1-plane graph}. The notion of $1$-planarity was introduced by Ringel \cite{Ringel.1965} while trying to simultaneously color the vertices and faces of a plane graph $G$ such that any pair of adjacent/incident elements receive different colors. Ringel \cite{Ringel.1965} showed that every 1-planar graph is 7-colorable, and Borodin \cite{Borodin.1984,Borodin.1995} improved it to the 6-colorability. Recently in 2017, Kobourov, Liotta and Montecchiani \cite{KLM2017} reviewed the current literature covering various research streams about 1-planarity, such as characterization and recognition, combinatorial properties, and geometric representations.

Clearly, every crossing $c$ in a 1-plane graph $G$ is generated by two mutually crossed edges $e_1$ and $e_2$. Thus, for every crossing $c$ there exists a vertex set $M_G(c)$ of size four, where $M_G(c)$, the \emph{cluster} of $c$, consists of the end-vertices of $e_1$ and $e_2$. For two distinct crossings $c_1$ and $c_2$ in a 1-plane graph $G$, it is clear that $|M_G(c_1)\cap M_G(c_2)|\leq 2$ (see \cite{Z2014}).

Let $G$ be a 1-plane graph. If $M_G(c_1)\cap M_G(c_2)=\emptyset$ for any two distinct crossings $c_1$ and $c_2$, then $G$ is a \emph{plane graph with independent crossings} (\emph{IC-plane graph}, for short). A graph that admits a drawing homeomorphic to an IC-plane graph is an \emph{IC-planar graph}. The IC-planarity was first considered by Albertson \cite{Albertson.2008} in 2008, who conjectured that every IC-planar graph is 5-colorable. This conjecture was confirmed by Kr\'al and Stacho \cite{KS2010} in 2010. Note that IC-planar graph can be non-planar.

In this paper, we consider the equitable partition problem of IC-planar graphs by proving the following.

\begin{thm}\label{1}
Every plane graph with independent crossings and with girth at least $g$ has an equitable partition into $m$ induced forests for any $m\geq \mathfrak{F}(g)$, where $$\mathfrak{F}(g)=\left\{
                                                    \begin{array}{ll}
                                                      8, & \hbox{if $g=3$;} \\
                                                      6, & \hbox{if $g=4$;} \\
                                                      5, & \hbox{if $g=5$;} \\
                                                      4, & \hbox{if $g=6$;} \\
                                                      3, & \hbox{if $g=26$.}
                                                    \end{array}
                                                  \right.$$
\end{thm}

\section{Preliminaries}

If a graph $G$ has an equitable partition into $m$ induced forests, we say that $G$ is \emph{equitably tree-$m$-colorable}, and has an \emph{equitable tree-$m$-coloring}. Let $\mathcal{G}_g$ be the class of IC-plane graph with girth at least $g$. Note that $\mathcal{G}_3\supseteq \mathcal{G}_4 \supseteq  \mathcal{G}_5\supseteq \ldots\supseteq \mathcal{G}_{+\infty}$.

\begin{lem}\label{2}
If $G\in \mathcal{G}_g$, then $$e(G)\leq \frac{5g-2}{4g-8}|G|-\frac{2g}{g-2}.$$
\end{lem}

\begin{proof}
Since every IC-plane graph $G$ has at most $\lfloor\frac{1}{4}|G|\rfloor$ crossings by its definition, we can obtain a plane graph $G'$ with order $|G|$ via removing at most $\lfloor\frac{1}{4}|G|\rfloor$ edges from $G$. Since $g(G')\geq g(G)\geq g$, $e(G')\leq \frac{g}{g-2}\big(|G|-2\big)$ by the famous Euler's formula. Therefore, the required result holds since $e(G)\leq e(G')+\frac{1}{4}|G|$.
\end{proof}

\begin{lem}\label{3}
Let $G$ be a graph in $\mathcal{G}_g$.\\
(a) If $g=3$, then $\delta(G)\leq 6$;\\
(b) if $g=4$, then $\delta(G)\leq 4$;\\
(c) if $g\geq 5$, then $\delta(G)\leq 3$.
\end{lem}

\begin{proof}
  The average degree $\overline{d}(G)$ of $G$ is $2e(G)/|G|$, and thus is at most $ \frac{5g-2}{2g-4}$ by Lemma \ref{2}. If $g=3$, then $\overline{d}(G)\leq 6.5$. If $g=4$, then $\overline{d}(G)\leq 4.5$. If $g\geq 5$, then $\overline{d}(G)<4$. Since $\delta(G)\leq \lfloor\overline{d}(G)\rfloor$, the results hold immediately.
\end{proof}

\begin{lem}\label{4}
Let $m\geq \mathfrak{G}(g)$ be a fixed integer, where $$\mathfrak{G}(g)=\left\{
                                                    \begin{array}{ll}
                                                      5, & \hbox{if $g=3$;} \\
                                                      3, & \hbox{if $g\geq 4$.}
                                                    \end{array}
                                                  \right.$$
If every graph in $\mathcal{G}_g$ of order $mt$ is equitably tree-$m$-colorable for any integer $t\geq 1$, then every graph in $\mathcal{G}_g$ is equitably tree-$m$-colorable.
\end{lem}

\begin{proof}
Let $G$ be a graph in $\mathcal{G}_g$ with order $n$.
If $n\leq m$, then it is trivial that $G$ is equitably tree-$m$-colorable. Hence we assume that $n>m$, and next prove this lemma by induction on $n$ (assuming that the result holds for graphs in $\mathcal{G}_g$ with order less than $n$).

If $n$ is divisible by $m$, then the required result holds directly. Hence we assume that $mt<n<m(t+1)$ and $t\geq 1$ is an integer.

Let $v\in V(G)$ be a vertex with minimum degree. By the induction hypothesis, $G-v$ has an equitable tree-$m$-coloring $\phi$. Let $V_{1},~V_{2},\cdots,~V_{m}$ be the color classes of $\phi$, where $|V_{i}|=t$ or $t+1$ for all $i\geq1$.

If $n=m(t+1)-1$, then we add an isolated vertex $v$ to $G$. Clearly, the resulting graph $G'$ is an IC-plane graph of order $m(t+1)$. By the condition of this lemma, $G'$ has an equitable tree-$m$-coloring such that all color classes have the same size. Removing $v$ from $G'$, we obtain the graph $G$ with an equitable tree-$m$-coloring.

Hence in the following, we assume that $n\leq m(t+1)-2$.
Since $|G-v|=n-1\leq m(t+1)-3$, among $V_{1},~V_{2},\cdots,~V_{m}$ there are at most $m-3$ classes containing exactly $t+1$ vertices.

If $g=3$, then $d(v)\leq 6$ by Lemma \ref{3}(a). Therefore, there are at least $m-3$ color classes among $V_{1},~V_{2},~\cdots,~V_{m}$ satisfying $|N(v)\bigcap V_{i}|\leq1$. Without loss of generality, assume that $|N(v)\bigcap V_{i}|\leq1$ for all $4\leq i\leq m$. If $|V_{i}|=t$ for some $i\geq4$, then by adding $v$ to $V_{i}$, we get an equitable tree-$m$-coloring of $G$ (with color classes $V_{1},~\cdots,~V_{i-1},~V_{i}\bigcup \{v\},~V_{i+1},~\cdots,~V_{m})$. Hence we assume that $|V_{i}|=t+1$ for all $i\geq4$. This implies that $|V_1|=|V_2|=|V_3|=t$, since $|G-v|\leq m(t+1)-3$.

If there exists $u\in \bigcup_{i=4}^{m}V_{i}$ such that $e(u,V_{j})\leq1$ for some $1\leq j\leq3$, then by transferring $v$ to the color class containing $u$, and adding $u$ to $V_{j}$, we get an equitable tree-$m$-coloring of $G$. Hence, for any $u\in \bigcup_{i=4}^{m}V_{i}$ and any $V_{j}$ with $1\leq j\leq3$, we have $e(u,V_{j})\geq2$. This implies $e(G)\geq6(t+1)(m-3)$. Since $G[E(V_{1}\bigcup V_{2}\bigcup V_{3},\bigcup_{i=4}^{m}V_{i})]$ is a bipartite IC-plane graph (so it has girth at least 4), $e(G)\leq\frac{9}{4}[m(t+1)-2]-4$ by Lemma \ref{2}. Hence $\frac{9}{4}[m(t+1)-2]-4\geq6(t+1)(m-3)$. But, this is a contradiction for $m\geq5$.

If $g\geq 4$, then $d(v)\leq 4$ by Lemma \ref{3}(b). Therefore, there are at least $m-2$ color classes among $V_{1},~V_{2},~\cdots,~V_{m}$ satisfying $|N(v)\bigcap V_{i}|\leq1$. Without loss of generality, assume that $|N(v)\bigcap V_{i}|\leq1$ for all $3\leq i\leq m$.
Since $|G-v|\leq m(t+1)-3$, among $V_{3},~\cdots,~V_{m}$, there is at least one class, say $V_3$, containing exactly $t$ vertices. Therefore, by moving $v$ to $V_3$, we obtain an equitable tree-$m$-coloring of $G$.
\end{proof}

\section{The structures of the edge-minimal counterexample}

Let $G$ be an edge-minimal graph with $|G|=mt$ in the class $\mathcal{G}_g$  that is not equitably tree-$m$-colorable. Here we assume that $m\geq 8$ if $g=3$, $m\geq 6$ if $g=4$, $m\geq 5$ if $g=5$, $m\geq 4$ if $g=6$, and $m\geq 3$ if $g\geq 7$.
This section is devoted to exploring the structures of $G$, which will be later used to prove
Theorem \ref{1} by contradiction in the next section.

Clearly, $G$ contains a vertex of degree at least 1. Let $$\delta(g)=\left\{
                                                    \begin{array}{ll}
                                                      6, & \hbox{if $g=3$;} \\
                                                      4, & \hbox{if $g=4$;} \\
                                                      3, & \hbox{if $g\geq 5$.}
                                                    \end{array}
                                                  \right.$$
Since $\delta(G)\leq \delta(g)$ by Lemma \ref{3}, there is an edge $xx_1\in E(G)$ with $1\leq d(x)\leq \delta(g)$. By the minimality of $G$, $G-xx_1$ admits an equitable tree-$m$-coloring with $m$ color classes $V_1,V_2,\ldots,V_m$, each of which has size $t$.

Clearly, $xx_1$ is contained in a cycle of the subgraph induced by some color class, for otherwise the current coloring of $G-xx_1$ is just an equitable tree-$m$-coloring of $G$. Therefore, $x,x_1$ and another neighbor of $x$, say $x_2$, is contained in a same color class, say $V_1$, and then we assume that $N(x)\subseteq \cup_{i=1}^{\delta(g)-1}V_i$. Let $V'_1=V_1\setminus \{x\}$.

If $g=3$, then $d(x)\leq 6$. Since $x$ has two neighbors contained in $V_1$, among $V_2,V_3,V_4$ and $V_5$, at most two of them contains at least two neighbors of $x$. Hence we assume, without loss of generality, that $|N(x)\cap V_4|\leq 1$ and $|N(x)\cap V_5|\leq 1$.

If $g=4$, then $d(x)\leq 4$. Since $x$ has two neighbors contained in $V_1$, among $V_2$ and $V_3$, at most one of them contains at least two neighbors of $x$. Hence we assume, without loss of generality, that $|N(x)\cap V_3|\leq 1$.

If $g\geq 5$, then $d(x)\leq 3$. Since $x$ has two neighbors contained in $V_1$, $|N(x)\cap V_2|\leq 1$.

\begin{clm}\label{5}
(a) If $G\in \mathcal{G}_3$, then $e(v,V'_1)\geq 2$ for every $v\in \cup_{i=4}^m V_{i}$.\\
(b) If $G\in \mathcal{G}_4$, then $e(v,V'_1)\geq 2$ for every $v\in \cup_{i=3}^m V_{i}$.\\
(c) If $G\in \mathcal{G}_g$ with $g\geq 5$, then $e(v,V'_1)\geq 2$ for every $v\in \cup_{i=2}^m V_{i}$.
\end{clm}

\begin{proof}
  We just prove (a), and another two results can be similarly verified.
Suppose, to the contrary, that there exists $v\in V_{i}$ for some $i\geq4$ such that $e(v,V'_{1})\leq1.$ By transferring $v$ from $V_{i}$ to $V'_{1}$ and adding $x$ to $V_{i}\setminus \{v\}$, we get an equitable tree-$m$-coloring of $G$, a contradiction.
\end{proof}

\begin{clm}\label{6}
(a) If $G\in \mathcal{G}_3$ and $m\geq5$, then for every $v\in V_2\cup V_3$, $e(v,V'_1)\geq 2$;\\
(b) If $G\in \mathcal{G}_4$ and $m\geq5$, then for every $v\in V_2$, $e(v,V'_1)\geq 2$.
\end{clm}

\begin{proof}
 We just prove (a). Note that (b) is a corollary of (a) since $\mathcal{G}_4\subseteq \mathcal{G}_3$.

Suppose, to the contrary, that there exists $w\in V_{2}$ such that $e(w,V'_{1})\leq1$.
 In this case,
\begin{align}\label{7}
&e(v,V_{2})\geq2 ~\text{for each} ~v\in\bigcup_{i=4}^{m}V_{i}.
\end{align}
Otherwise, suppose that $e(v,V_{2})\leq1$ for some $v\in V_{i}$ with $4\leq i\leq m$. Transferring $v$ from $V_{i}$ to $V_{2}$, $w$ from $V_{2}$ to $V'_{1}$ and adding $x$ to $V_{i}\setminus \{v\}$, we get an equitable tree-$m$-coloring of $G$, a contradiction.

 If there exists $w'\in V_{3}$ such that $e(w',V_{2})\leq1$, then $e(v,V_{3})\geq2$ for each $v\in\bigcup_{i=4}^{m}V_{i}$. Otherwise, suppose that $e(v,V_{3})\leq1$ for some $v\in V_{i}$ with $4\leq i\leq m$. Transferring $v$ from $V_{i}$ to $V_{3}$, $w'$ from $V_{3}$ to $V_{2}$, $w$ from $V_{2}$ to $V'_{1}$ and adding $x$ to $V_{i}\setminus \{v\}$, we get an equitable tree-$m$-coloring of $G$, a contradiction.

 If there exists $w'\in V_{3}$ such that $e(w',V'_{1})\leq1$, then $e(v,V_{3})\geq2$ for each $v\in\bigcup_{i=4}^{m}V_{i}$. Otherwise, suppose that $e(v,V_{3})\leq1$ for some $v\in V_{i}$ with $4\leq i\leq m$. Transferring $v$ from $V_{i}$ to $V_{3}$, $w'$ from $V_{3}$ to $V'_{1}$ and adding $x$ to $V_{i}\setminus \{v\}$, we get an equitable tree-$m$-coloring of $G$, a contradiction.

 In each of the above two cases, by Claim \ref{5} and by \eqref{7}, we have $e(\bigcup_{i=4}^{m}V_{i},V'_{1}\bigcup V_{2}\bigcup V_{3})\geq6(m-3)t$. Since $G[E(\bigcup_{i=4}^{m}V_{i},V'_{1}\bigcup V_{2}\bigcup V_{3})]$ is a bipartite IC-plane graph of order $mt-1$, we have $e(\bigcup_{i=4}^{m}V_{i},V'_{1}\bigcup V_{2}\bigcup V_{3})\leq\frac{9}{4}(mt-1)-4=\frac{9}{4}mt-\frac{25}{4}$ by Lemma \ref{2}. Since $m\geq5$, $6(m-3)t>\frac{9}{4}mt-\frac{25}{4}$, a contradiction, too.
  Hence
 \begin{align}\label{8}
 &e(w',V_{2})\geq2~\text{and}~e(w',V'_{1})\geq2~\text{for each}~ w'\in V_{3}.
 \end{align}

     By Claim \ref{5}, \eqref{7}, and \eqref{8}, we conclude that $e(\bigcup_{i=3}^{m}V_{i},V'_{1}\bigcup V_{2})\geq4(m-2)t$.

     Since $G[E(\bigcup_{i=3}^{m}V_{i},V'_{1}\bigcup V_{2})]$ is a bipartite IC-plane graph of order $mt-1$,  $e(\bigcup_{i=3}^{m}V_{i},V'_{1}\bigcup V_{2})\leq\frac{9}{4}(mt-1)-4=\frac{9}{4}mt-\frac{25}{4}$ by Lemma \ref{2}. Since $m\geq5$, $4(m-2)t>\frac{9}{4}mt-\frac{25}{4}$, a contradiction. Hence, $e(w,V'_{1})\geq2$ for each $w\in V_{2}$. 
     By similar argument as above, we conclude that $e(w',V'_{1})\geq2$ for each $w'\in V_{3}$.
\end{proof}

Let $A=\cup_{i=2}^m V_i$. By Claims \ref{5} and \ref{6}, if $G\in \mathcal{G}_3$ and $m\geq 5$, or $G\in \mathcal{G}_5$, then
\begin{align}\label{10}
&e(v,V'_1)\geq 2~ \text{for every} ~v\in A,~\text{and thus}~e(A,V'_1)\geq 2(m-1)t.
 \end{align}
 Therefore, we divide $A$ into two parts, say $A_1$ and $A\setminus A_1$, where $A_1=\{v\in A~|~e(v,V'_1)=2\}$. Let $r=|A_1|$, then
 \begin{align}\label{11}
 &e(A,V'_{1})\geq2r+3\big((m-1)t-r\big)=3(m-1)t-r.
  \end{align}

Next, we calculate the lower bound for $r$.
Since $G[E(A,V'_{1})]$ is a bipartite IC-plane graph (so odd cycles are forbidden) and is also a subgraph of $G$, its girth $g_0$ is an even integer no less than $g$. Hence $g_0\geq 4$ if $g\leq 4$, and $g_0\geq 6$ if $g\geq 5$.


By \eqref{11} and Lemma \ref{2}, $$\frac{5g_0-2}{4g_0-8}(mt-1)-\frac{2g_0}{g_0-2}\geq e(A,V'_{1}) \geq 3(m-1)t-r,$$ which implies that
  \begin{align}\label{12}
  &r\geq\left\{
          \begin{array}{ll}
            \big(\frac{3}{4}m-3\big)t+\frac{25}{4}, & \hbox{if $g=3$ or $g=4$;} \\[.4em]
            \big(\frac{5}{4}m-3\big)t+\frac{19}{4}, & \hbox{if $g\geq 5$.}
          \end{array}
        \right.
\end{align}

\begin{lem}\label{13}
There exists a vertex $z\in V'_1$ that has two nonadjacent neighbors $y_1,y_2$ in $A_1$ if one of the following conditions is satisfied:

(i) $r> 2(t-1)$ and $g=3$;

(ii) $r> \frac{1}{2}(t-1)$ and $g\geq 4$;

(iii) $e(A,V'_1)\leq (3m-5)t+1$ and $g=3$;


(iv) $G\in \mathcal{G}_3$ and $m\geq 7$;

(v) $G\in \mathcal{G}_4$ and $m\geq 5$;

(vi) $G\in \mathcal{G}_5$ and $m\geq 3$.

\end{lem}

\begin{proof}
(i)
Suppose that for each vertex $z\in V'_{1}$, $e(z,A_{1})\leq4$. Since $|V'_{1_{}}|=t-1$ and $r>2(t-1)$, $4(t-1)\geq e(A_{1},V'_{1})=2r>4(t-1)$, a contradiction.  Thus there exists a vertex $z\in V'_{1}$, such that $e(z,A_{1})\geq5$. Since $K_{6}$ is not an IC-plane graph, there are two neighbors of $z$ in $A_{1}$ that are not adjacent, and thus the required structure occurs.

(ii)
Suppose that for each vertex $z\in V'_{1}$, $e(z,A_{1})\leq1$. Since $|V'_{1_{}}|=t-1$ and $r>\frac{1}{2}(t-1)$, $(t-1)\geq e(A_{1},V'_{1})=2r>(t-1)$, a contradiction.  Thus there exists a vertex $z\in V'_{1}$, such that $e(z,A_{1})\geq2$. Since $K_{3}$ is forbidden in an IC-plane graph with girth at least 4, there are two neighbors of $z$ in $A_{1}$ that are not adjacent, and thus the required structure occurs.

(iii)
In this case, by \eqref{11}, $(3m-5)t+1\geq e(A,V'_{1})\geq3(m-1)t-r$, which implies $r>2(t-1)$. Hence by (i), we complete the proof.

(iv)
If $m\geq7$, then by \eqref{12}, $r\geq(\frac{3}{4}\times 7-3)t+\frac{25}{4}>2(t-1)$ and (i) is satisfied.

(v)
If $m\geq5$, then by \eqref{12}, $r\geq(\frac{3}{4}\times 5-3)t+\frac{25}{4}>\frac{1}{2}(t-1)$ and (ii) is satisfied.

(vi)
If $m\geq3$, then by \eqref{12}, $r\geq(\frac{5}{4}\times 3-3)t+\frac{19}{4}>\frac{1}{2}(t-1)$ and (ii) is satisfied.
\end{proof}

Suppose that there exists a vertex $z\in V'_1$ that has two nonadjacent neighbors $y_1,y_2$ in $A_1$. It is easy to see that $V'_1\cup \{y_1,y_2\}\setminus \{z\}$ induces a forest $F_1$ of order $t$. Let $G'$ be the graph induced by $A\cup \{x,z\}\setminus \{y_1,y_2\}$. Note that $|G'|=|A|-2+2=(m-1)t$.

\begin{clm}\label{14}
$e(G')\leq e(G)-(m-1)t-2$.
\end{clm}

\begin{proof}

Since $e(v,V'_1)\geq 2$ for every $v\in A$, we have $e(v,V'_1\setminus \{z\})\geq 1$ for every $v\in A\setminus \{y_1,y_2\}$ and $e(A\setminus \{y_1,y_2\},V'_1\setminus \{z\})\geq|A\setminus \{y_1,y_2\}|=(m-1)t-2$.
Counting the four edges $xx_1,xx_2,zy_1,zy_2$, we immediately have $e(G',F_1)\geq (m-1)t-2+4=(m-1)t+2$. This implies that $e(G')\leq e(G)-(m-1)t-2$.
\end{proof}

\begin{clm}\label{15}
If $G'$ is equitably tree-$(m-1)$-colorable, then $G$ is equitably tree-$m$-colorable.
\end{clm}

\begin{proof}
 Since $|G'|=(m-1)t$, $G'$ has an equitable partition into $m-1$ induced forests $F_2,\ldots,F_m$ with $|F_{i}|=t$ for each $2\leq i\leq m$.
 It follows that $G$ has an equitable partition into $m$ induced forests $F_1,F_2,\ldots,F_m$, a contradiction to the choice of $G$. Recall that $F_{1}$ is the graph induced by $V'_{1}\bigcup \{y_{1},y_{2}\}\backslash \{z\}$, which is a forest of order $t$.
\end{proof}

\section{The proof of Theorem \ref{1}}

In the proofs of the following theorems, we use the edge-minimal-counterexample-arguments as mentioned in Section 3, and thus the notations and results in Section 3 can be applied here.

\begin{thm}\label{girth-3}
 Let $s\in \{5,6,7,8\}$. If $G$ is a graph in $\mathcal{G}_3$ of order $mt$ and size at most $$\bigg(\frac{4s-19}{4}m+\frac{-2s^2+17s-8}{4}\bigg)t-(22-2s),$$ then $G$ has an equitable partition into $m$ induced forests for any $m\geq s$.
\end{thm}

\begin{proof}

 We prove it by induction on $s$. First of all, if $s=5$, then $e(G)\leq(\frac{1}{4}m+\frac{27}{4})t-12$, which implies by \eqref{10} that $e(A,V'_{1})-e(G)\geq2(m-1)t-\big((\frac{1}{4}m+\frac{27}{4})t-12\big)>0$ for $m\geq5$, a contradiction.

 We assume that the result holds for $s=k-1$, where $6\leq k\leq8$. Now we consider the case when $s=k$.

 If $s\in \{7,8\}$, then by Lemma \ref{13}(iv), there exists a vertex $z\in V'_{1}$ that has two nonadjacent neighbors $y_{1}$, $y_{2}$ in $A_{1}$. If $s=6$, then
 $e(A,V'_{1})\leq e(G)\leq (\frac{5}{4}m+\frac{11}{2})t-10\leq(3m-5)t+1$,
 and thus by Lemma $\ref{13}$(iii) the same result holds. Therefore, by Claim \ref{14}, we have
  \begin{align*}
  e(G')&\leq e(G)-(m-1)t-2\\
  &\leq\bigg(\frac{4k-19}{4}m+\frac{-2k^2+17k-8}{4}\bigg)t-(22-2k)-(m-1)t-2\\
  &=\bigg(\frac{4(k-1)-19}{4}(m-1)+\frac{-2(k-1)^2+17(k-1)-8}{4}\bigg)t-\big(22-2(k-1)\big).
  \end{align*}

 Since $G'$ is an IC-plane graph and $m-1\geq s-1=k-1$, $G'$ admits an equitable tree-$(m-1)$-coloring by the induction hypothesis.
 Hence by Claim \ref{15}, $G$ admits an equitable tree-$m$-coloring.
\end{proof}

\begin{thm}\label{girth-4}
 Let $s\in \{4,5,6\}$. If $G$ is a graph in $\mathcal{G}_4$ of order $mt$ and size at most $$\bigg(\frac{4s-15}{4}m+\frac{-2s^2+13s-6}{4}\bigg)t-(16-2s),$$ then $G$ has an equitable partition into $m$ induced forests for any $m\geq s$.
\end{thm}

\begin{proof}
We prove it by induction on $s$. First of all, if $s=4$, then $e(G)\leq(\frac{1}{4}m+\frac{7}{2})t-8$, which implies by \eqref{10} that $e(A,V'_{1})-e(G)\geq2(m-1)t-\big((\frac{1}{4}m+\frac{7}{2})t-8\big)>0$ for $m\geq4$, a contradiction.

 We assume that the result holds for $s=k-1$, where $5\leq k\leq6$. Now we consider the case when $s=k$.

 If $s\in \{5,6\}$, then by Lemma $\ref{13}$(v) there exists a vertex $z\in V'_{1}$ that has two nonadjacent neighbors $y_{1}$, $y_{2}$ in $A_{1}$. Therefore, by Claim \ref{14}, we have
   \begin{align*}
  e(G')&\leq e(G)-(m-1)t-2\\
  &\leq\bigg(\frac{4k-15}{4}m+\frac{-2k^2+13k-6}{4}\bigg)t-(16-2k)-(m-1)t-2\\
  &=\bigg(\frac{4(k-1)-15}{4}(m-1)+\frac{-2(k-1)^2+13(k-1)-6}{4}\bigg)t-\big(16-2(k-1)\big).
  \end{align*}

Since $G'$ is an IC-plane graph and $m-1\geq s-1=k-1$, $G'$ admits an equitable tree-$(m-1)$-coloring by the induction hypothesis.
 Hence by Claim \ref{15}, $G$ admits an equitable tree-$m$-coloring.
\end{proof}

Choosing $s$ to be 8 and 6 in Theorem \ref{girth-3} and in Theorem \ref{girth-4}, respectively, we conclude by Lemmas \ref{2} and \ref{4} that

\begin{thm}\label{girth-3-end}
Every plane graph with independent crossings has an equitable partition into $m$ induced forests for each $m\geq 8$. \hfill$\square$
\end{thm}

\begin{thm}\label{girth-4-end}
Every plane graph with independent crossings and with girth at least 4 has an equitable partition into $m$ induced forests for each $m\geq 6$. \hfill$\square$
\end{thm}

Now, we consider plane graph with independent crossings and with higher girth.

\begin{thm}\label{girth-5-end}
Every plane graph with independent crossings and with girth at least 5 has an equitable partition into $m$ induced forests for each $m\geq 5$.
\end{thm}

\begin{proof}
By Lemma \ref{4}, we assume that the order of the considering graph $G$ is divided by $m$, that is, $|G|=mt$. By Lemma \ref{2}, we have $$e(G)\leq\frac{23}{12}mt-\frac{10}{3}.$$

Since $m\geq 5$, there exists a vertex $z\in V'_{1}$ that has two nonadjacent neighbors $y_{1}$, $y_{2}$ in $A_{1}$ by Lemma $\ref{13}$(vi). Hence by Claim \ref{14},
 \begin{align}\label{16}
  \notag e(G')&\leq e(G)-(m-1)t-2\\
  \notag &\leq\frac{23}{12}mt-\frac{10}{3}-(m-1)t-2\\
  &=\bigg(\frac{11}{12}m+1\bigg)t-\frac{16}{3}.
  \end{align}

Now, by Claim \ref{15}, proving that $G'$ admits an equitable tree-$(m-1)$-coloring is enough. Applying the edge-minimum-counterexample-arguments to $G'$, we immediately have, by \eqref{10}, that
$$e(G')\geq 2\big((m-1)-1\big)t=(2m-4)t.$$
Hence by \eqref{16}, we have $$\bigg(\frac{11}{12}m+1\bigg)t-\frac{16}{3}\geq (2m-4)t,$$ which implies that $m\leq 4$, a contradiction.
\end{proof}

\begin{thm}\label{girth-6-end}
Every plane graph with independent crossings and with girth at least 6 has an equitable partition into $m$ induced forests for each $m\geq 4$.
\end{thm}

\begin{proof}
By Lemma \ref{4}, we assume that the order of the considering graph $G$ is divided by $m$, that is, $|G|=mt$. By Lemma \ref{2}, we have $$e(G)\leq\frac{7}{4}mt-3.$$

Since $m\geq 4$, there exists a vertex $z\in V'_{1}$ that has two nonadjacent neighbors $y_{1}$, $y_{2}$ in $A_{1}$ by Lemma $\ref{13}$(vi). Hence by Claim \ref{14},
 \begin{align}\label{17}
  \notag e(G')&\leq e(G)-(m-1)t-2\\
  \notag &\leq\frac{7}{4}mt-3-(m-1)t-2\\
  &=\bigg(\frac{3}{4}m+1\bigg)t-5.
  \end{align}

Now, by Claim \ref{15}, proving that $G'$ admits an equitable tree-$(m-1)$-coloring is enough. Applying the edge-minimum-counterexample-arguments to $G'$, we immediately have, by \eqref{10}, that
$$e(G')\geq 2\big((m-1)-1\big)t=(2m-4)t.$$
Hence by \eqref{17}, we have $$\bigg(\frac{3}{4}m+1\bigg)t-5\geq (2m-4)t,$$ which implies that $m\leq 3$, a contradiction.
\end{proof}

\begin{thm}\label{girth-26-end}
Every plane graph with independent crossings and with girth at least 26 has an equitable partition into $m$ induced forests for each $m\geq 3$.
\end{thm}

\begin{proof}
By Lemma \ref{4}, we assume that the order of the considering graph $G$ is divided by $m$, that is, $|G|=mt$. By Lemma \ref{2}, we have $$e(G)\leq\frac{4}{3}mt-\frac{13}{6}.$$
Hence by \eqref{10} we conclude that $e(A,V'_{1})-e(G)\geq2(m-1)t-(\frac{4}{3}mt-\frac{13}{6})>0$ for $m\geq 3$, a contradiction. It follows by the edge-minimum-counterexample-arguments that $G$ admits an equitable tree-$m$-coloring.
\end{proof}

\proof[The proof of Theorem \ref{1}]

See Theorems \ref{girth-3-end}, \ref{girth-4-end}, \ref{girth-5-end}, \ref{girth-6-end} and \ref{girth-26-end}, respectively. \hfill$\square$

\section{Remarks}

 Formerly, the minimum integer $k$ such that $G$ has an equitable partition into $k$ induced forests is the \emph{equitable vertex arboricity} of $G$, denoted by $va_{eq}(G)$, and the minimum integer $k$ such that $G$ has an equitable partition into $m$ induced forests for any $m\geq k$ is the \emph{equitable vertex arborable threshold} of $G$, denoted by $va_{eq}^*(G)$. Theorem \ref{1} actually implies that $va_{eq}^*(G)\leq \mathfrak{F}(g)$ if $G$ is a plane graph with independent crossings and with girth at least $g$.
 Precisely, choosing $g=3$, we conclude that $va_{eq}^*(G)\leq 8$ if $G$ is a plane graph with independent crossings. Here, we do not know whether the upper bound 8 for $va_{eq}^*(G)$ is sharp (actually we think that it may be improved), but this bound is acceptable at this stage, since 8 is a constant not very large. Note that the paper of Esperet, Lemoine and Maffray \cite{ELM.2015} implies that $va_{eq}^*(G)\leq 19$ if $G$ is a 1-planar graph, whose acyclic chromatic number is at most 20 \cite{Borodin.1999}.

 In 2013, Wu, Zhang and Li \cite{WZL.2013} put forward two conjectures in their paper. Although Esperet, Lemoine and Maffray \cite{ELM.2015} solved one in 2015, the other (Conjecture \ref{conj1}) is still open.

 \begin{conj}\label{conj1}
   $va_{eq}^*(G)\leq \lceil\frac{\Delta(G)+1}{2}\rceil$ for any simple graph $G$.
 \end{conj}

 As far as we know, Conjecture \ref{conj1} has been verified for complete graphs \cite{WZL.2013}, balanced complete bipartite graphs \cite{WZL.2013}, graphs with maximum degree $\Delta\geq (|G|-1)/2$ \cite{ZW.2014,ZN.2019}, graphs with maximum degree $\Delta\leq 3$ \cite{Z.2016}, 5-degenerate graphs (so graphs with maximum degree $\Delta\leq 5$) \cite{CGSWW.2017}, and $d$-degenerate graphs with maximum degree $\Delta\geq 10d$ \cite{ZNLL.2019}.

Looking back to Theorem \ref{1}, we immediately find that Conjecture \ref{conj1} holds for any plane graph with independent crossings and with maximum degree at least 14. Of course, we may do not like the lower bound 14 for the maximum degree there. If we can pull this bound down to 6, then Conjecture \ref{conj1} holds for all plane graphs with independent crossings.

Note that every plane graph with independent crossings is 6-degenerate (a graph is \emph{$k$-degenerate} if $\delta(H)\leq k$ for any $H\subseteq G$). Therefore, an alternate task is to prove Conjecture \ref{conj1} for all $6$-degenerate graphs directly. Actually, we propose the following conjecture (also see \cite{LZ.2019}).

 \begin{conj}\label{conj2}
   $va_{eq}^*(G)\leq k$ for any $k$-degenerate graph $G$.
 \end{conj}

 If Conjecture \ref{conj2} can be verified, then the bound $k$ for $va_{eq}^*(G)$ is sharp. This fact can be seen from the graph $G$ obtained from $K_{k}$ via adding $t\geq 2k-3$ vertices, each of which is adjacent to all vertices of $K_{k}$. Clearly, $G$ is $k$-degenerate.

 If $va_{eq}^*(G)\leq k-1$, then $G$ has an equitable tree-$(k-1)$-coloring $\varphi$. Under this coloring, two vertices of $K_{k}$ shall receive the same color, say 1, and all but these two vertices are not colored with 1, because otherwise a monochromatic triangle appears. Since $\varphi$ is equitable, each of the colors in $\{2,3,\ldots,k-1\}$ appears at most three times in $G$. This implies that there are at most $2+3(k-2)=3k-4$ colored vertices, contradicting the fact that $|G|=k+t\geq 3k-3$.

Let $d$ be a positive integer. An equitable \emph{$d$-defective tree-$k$-coloring} of a graph $G$ is an equitable tree-$k$-coloring of $G$ such that the subgraph induced by each color class has maximum degree at most $d$.

The minimum integer $k$ such that $G$ has an equitable $d$-defective tree-$k$-coloring
is the \emph{equitable vertex $d$-arboricity} of $G$, denoted by $va_{eq}^d(G)$, and the minimum integer $k$ such that $G$ has an equitable $d$-defective tree-$m$-coloring
for any $m\geq k$ is the \emph{equitable vertex $d$-arborable threshold} of $G$, denoted by $va_{eq}^{*d}(G)$.
In 2011, Fan \emph{et al.}\,\cite{FKLMWZ.2011} prove that $va_{eq}^{*1}(G)\leq \Delta(G)$ for any graph $G$. Recently, Zhang and Niu \cite{ZN.2019} proved that $va_{eq}^{*2}(G)\leq \lceil\frac{\Delta(G)+1}{2}\rceil$ if $G$ is a graph with $\Delta(G)\geq (|G|-1)/2$.

In the paper \cite{ELM.2015}, Esperet, Lemoine and Maffray mentioned (pointed out by Yair Caro, actually) that there does not exist a constant $c$ so that $va_{eq}^{*2}(G)\leq c$ for any planar graph $G$.
The outer-planar graph obtained from a large path by adding a universal vertex is an example supporting this conclusion.

In fact, one can easily show for any fixed integer $d\geq 1$ that $va_{eq}^{d}(G)=va_{eq}^{*d}(G)=\lceil{\frac{\Delta+1}{d}}\rceil$ if $G$ is a star with maximum degree $\Delta$. Hence, for any fixed integer $d\geq 1$, finding a constant $m$ such that every planar graph (even for outer-planar graph) has an equitable partition into $m$ induced forests with maximum degree at most $d$ is impossible. From this point of view, the ``constant'' results on the equitable vertex arboricity or the equitable vertex arborable threshold ($d=+\infty$) of planar graphs and its relative classes are very interesting.

%


\end{document}